\documentclass[11pt]{article}

\usepackage{amssymb}
\usepackage{amsmath}
\usepackage{fullpage}
\usepackage{amsthm, mathrsfs}
\usepackage{latexsym}
\usepackage[all]{xy}
\usepackage{graphicx}

\newcommand{\R}{\mathbb R}
\newcommand{\C}{\mathbb C}
\newcommand{\Q}{\mathbb Q}
\newcommand{\Z}{\mathbb Z}
\newcommand{\N}{\mathbb N}
\newcommand{\hooklongrightarrow}{\lhook\joinrel\longrightarrow}

\def\H{{\mathbb H}}

\def\L{{\mathbb L}}
\def\I{{\mathbb I}}

\newtheorem{theorem}{Theorem}[section]
\newtheorem{lemma}[theorem]{Lemma}
\newtheorem{definition}[theorem]{Definition}

\newtheorem{proposition}[theorem]{Proposition}
\newtheorem{corollary}[theorem]{Corollary}

\newtheorem{remark}[theorem]{Remark}

\newenvironment{acknowledgement}[1][Acknowledgements]
{\begin{trivlist} \item[\hskip \labelsep {\bfseries #1}]}
{\end{trivlist}}

\newenvironment{question}[1][\emph{\textbf{Question:}}]
{\begin{trivlist} \item[\hskip \labelsep {\bfseries #1}]}
{\end{trivlist}}

\title{Aeppli-Bott-Chern cohomology and Deligne cohomology from a viewpoint of Harvey-Lawson's spark complex}

\author{Jyh-Haur Teh}
\date{}
\begin{document}
\maketitle

\abstract{By comparing Deligne complex and Aeppli-Bott-Chern complex, we construct a differential cohomology $\widehat{H}^*(X, *, *)$ that plays the role of Harvey-Lawson spark group $\widehat{H}^*(X, *)$,
and a cohomology $H^*_{ABC}(X; \Z(*, *))$ that plays the role of Deligne cohomology $H^*_{\mathcal{D}}(X; \Z(*))$ for every complex manifold $X$. They fit in the short exact sequence
$$ 0\rightarrow H^{k+1}_{ABC}(X; \Z(p, q)) \rightarrow \widehat{H}^k(X, p, q) \overset{\delta_1}{\rightarrow} Z^{k+1}_I(X, p, q) \rightarrow 0$$
and $\widehat{H}^{\bullet}(X, \bullet, \bullet)$ possess ring structure and refined Chern classes, acted by the complex conjugation, and if some primitive cohomology groups of $X$ vanish, there is a Lefschetz isomorphism. Furthermore, the ring structure of $H^{\bullet}_{ABC}(X; \Z(\bullet, \bullet))$ inherited from $\widehat{H}^{\bullet}(X, \bullet, \bullet)$ is compatible with the one of the analytic Deligne cohomology $H^{\bullet}(X; \Z(\bullet))$. We compute
$\widehat{H}^*(X, *, *)$ for $X$ the Iwasawa manifold and its small deformations and get a refinement of the classification
given by Nakamura.}

\section{Introduction}
The theory of differential characters was founded by Cheeger and Simons (\cite{CS, BB, B}) around 1970. It obtains intensive development in the last 20 years. Physicists realize that
differential characters can be used in the mathematical formulation of generalized abelian gauge theories(\cite{F}), and mathematicians found that they appear naturally in
many mathematical problems (\cite{HL2, HS}). The interaction between physics and mathematics stimulates lot of development in both disciplines and the theory of differential characters is extended to various generalized differential cohomologies. The article \cite{BS} of Bunke and Schick gives us a nice overview about differential cohomologies, including differential K-theory, and their relation with physics, especially with string theory.

There are various constructions of differential cohomologies (\cite{BKS, BNV, HL2}).
A particular simple construction to us was given by Harvey and Lawson through their theory of spark complexes (\cite{H, HLZ, HL1, HL2, HZ, Z}) which unifies
many known results. By applying their theory, they constructed a $\overline{\partial}$-analogue (\cite{HL1, H}) of differential characters for complex manifolds.
The Harvey-Lawson spark group $\widehat{H}^k(X, p)$ of level $p$ of a complex manifold $X$ contains the analytic Deligne cohomology $H^{k+1}_{\mathcal{D}}(X, \Z(p))$ as a subgroup and fits in
the short exact sequence
$$0 \rightarrow H^{k+1}_{\mathcal{D}}(X, \Z(p)) \rightarrow \widehat{H}^k(X, p) \rightarrow \mathcal{Z}^{k+1}_{\Z}(X, p) \rightarrow 0$$
where $\mathcal{Z}^{k+1}_{\Z}(X, p)$ is the subgroup of complex differential $(k+1)$-forms with integral periods.
Deligne cohomology group $H^{k+1}_{\mathcal{D}}(X; \Z(p))$ is usually defined by the hypercohomology group $\H^{k+1}(X, \Z(p))$ of the
Deligne complex of sheaves
$$\Z(p): 0 \rightarrow \Z \hookrightarrow \Omega^0 \overset{d}{\rightarrow} \Omega^1 \overset{d}{\rightarrow} \cdots \overset{d}{\rightarrow} \Omega^{p-1} \rightarrow 0$$
where $\Omega^k$ is the sheaf of holomorphic $k$-forms.
Recall that the Aeppli and Bott-Chern cohomology (\cite{S}) of a complex manifold $X$ can be defined by the hypercohomology of the complex of sheaves:
$$B^{\bullet}_{p, q}:0 \rightarrow \C \rightarrow \mathcal{O}\oplus \overline{\mathcal{O}} \rightarrow \Omega^1\oplus \overline{\Omega}^1
\rightarrow \cdots \rightarrow \Omega^{p-1}\oplus \overline{\Omega}^{p-1}\rightarrow \overline{\Omega}^p \rightarrow \cdots \rightarrow
\overline{\Omega}^{q-1} \rightarrow 0$$ where $p\geq q$ and $\overline{\Omega}^k$ is the sheaf of anti-holomorphic $k$-forms. We have
$$H^{p, q}_A(X; \C)\cong \H^{p+q+1}(X, B^{\bullet}_{p+1, q+1}) \mbox{ and } H^{p, q}_{BC}(X; \C)\cong \H^{p+q}(X, B^{\bullet}_{p, q})$$
By this similarity to the definition of Deligne cohomology, it is natural to ask the following question

\begin{question}
Is there a differential cohomology that plays the role
of the Harvey-Lawson spark group $\widehat{H}(X, p)$, and the hypercohomology groups of the complex of sheaves
$$0 \rightarrow \Z \rightarrow \mathcal{O}\oplus \overline{\mathcal{O}} \rightarrow \Omega^1\oplus \overline{\Omega}^1
\rightarrow \cdots \rightarrow \Omega^{p-1}\oplus \overline{\Omega}^{p-1}\rightarrow \overline{\Omega}^p \rightarrow \cdots \rightarrow
\overline{\Omega}^{q-1} \rightarrow 0$$ that play the role of Deligne cohomology?
\end{question}

This is the motivation of this paper. We answer this question affirmative by constructing cohomology $\widehat{H}^*(X, p, q)$ and $H^*_{ABC}(X; \Z(p, q))$ for every complex manifold $X$, and integers $p, q\geq 0$
that fit in the short exact sequence:
$$ 0\rightarrow H^{k+1}_{ABC}(X; \Z(p, q)) \rightarrow \widehat{H}^k(X, p, q) \overset{\delta_1}{\rightarrow} Z^{k+1}_I(X, p, q) \rightarrow 0$$

The cohomology $\widehat{H}^{\bullet}(X, \bullet, \bullet)$ possess a ring structure and refined Chern classes, acted by the complex conjugation, and if some primitive cohomology groups
of $X$ vanish, there is a Lefschetz isomorphism. Furthermore, the ring structure of $H^{\bullet}_{ABC}(X; \Z(\bullet, \bullet))$ inherited from $\widehat{H}^{\bullet}(X, \bullet, \bullet)$ is compatible with the one of the analytic Deligne cohomology $H^{\bullet}(X; \Z(\bullet))$.
We compute
$\widehat{H}^{\bullet}(X, \bullet, \bullet)$ for $X$ the Iwasawa manifold and its small deformations and get a refinement of the classification
given by Nakamura. Such finer classification is different from the one given by Angella \cite{A}.

The paper is organized as follow. In section 2, we review Harvey and Lawson's theory of spark complexes and use it to construct $\widehat{H}^*(X, p, q)$. We show that the above mentioned sequence
is short exact, give a $3\times 3$-grid that relates Griffiths intermediate Jacobian and Hodge group and prove a Lefschetz property. In section 3, we establish a ring structure on $\widehat{H}^{\bullet}(X; \bullet, \bullet)$. In section 4, we construct refined Chern classes for
complex vector bundles and prove a Whitney product formula. Furthermore, for holomorphic vector bundles, we show that their refined Chern classes can be defined on $\widehat{H}^{\bullet}_{ABC}(X; \Z(\bullet, \bullet))$
and the total refined Chern class defines a natural map from holomorphic K-theory to $\widehat{H}^{\bullet}(X; \Z(\bullet, \bullet))$.
In section 5, we compute the $ABC$-cohomology of the Iwasawa manifold and its small deformations, and give a refinement 
of the classification given by Nakamura (\cite{Nak}).

\begin{acknowledgement}
The author thanks Siye Wu for his interest in this work and Taiwan National Center for Theoretical Sciences(Hsinchu) for proving a nice working environment.
\end{acknowledgement}
\section{Harvey-Lawson spark groups }
We recall the construction of spark groups given by Harvey and Lawson in \cite{HL1}.

\begin{definition}(Spark complexes)
Suppose that $F^{\bullet}=\oplus_{i\geq 0}F^i,
E^{\bullet}=\oplus_{i\geq 0}E^i, I^{\bullet}=\oplus_{i\geq 0}I^i$
are cochain complexes and $\Psi:I \rightarrow F$ is a morphism of
cochain complexes, $E^{\bullet} \hookrightarrow F^{\bullet}$ is an embedding with the
following properties:
\begin{enumerate}
\item $\Psi(I^k)\cap E^k=\{0\}$ for all $k>0$,
\item $\Psi:I^0 \rightarrow F^0$ is injective,
\item the embedding induces an isomorphism $H^k(E^{\bullet}) \rightarrow H^k(F^{\bullet})$.
\end{enumerate}
Then $\mathscr{S}=(F^{\bullet}, E^{\bullet}, I^{\bullet})$ is called a spark complex.
\end{definition}

\begin{definition}(Spark groups)
Given a spark complex $\mathscr{S}=(F^{\bullet}, E^{\bullet}, I^{\bullet})$, a spark of degree $k$ is a pair $(a, r)\in F^k\oplus I^{k+1}$ which satisfies the spark equation
$$\left\{
       \begin{array}{ll}
        da=e-\Psi(r) \mbox{ where }e\in E^{k+1}, \\
        dr=0.
        \end{array}
\right.$$

Let $\mathscr{S}^k(F^{\bullet}, E^{\bullet}, I^{\bullet})$ be the collection of all sparks of degree $k$ in $(F^{\bullet}, E^{\bullet}, I^{\bullet})$. Two sparks $(a, r), (a', r')$ of degree $k$ are
equivalent if there exists a pair $(b, s)\in F^{k-1}\oplus I^k$ such that
$$\left\{
\begin{array}{ll}
a-a'=db+\Psi(s),\\
r-r'=-ds.
\end{array}
\right.$$
We write $\widehat{H}^k(\mathscr{S})=S^k(F^{\bullet}, E^{\bullet}, I^{\bullet})/\sim$ for the group of equivalence spark classes of degree $k$.
\end{definition}

Let $X$ be a complex manifold of complex dimension $m$. We write $\mathcal{E}^{p, q}_{cpt}(X)$ for the space of $(p, q)$-forms with compact support on $X$. The space of currents of degree
$(p, q)$ on $X$ is the topological dual space $D'^{p, q}(X):=\{\mathcal{E}^{m-p, m-q}_{cpt}\}'$. We write

$$D'^k(X, p, q)=\underset{\underset{i_1<p}{i_1+j_1=k}}{\bigoplus} D'^{i_1, j_1}(X) \bigoplus \underset{\underset{j_2<q}{i_2+j_2=k}}{\bigoplus} D'^{i_2, j_2}(X)$$
and the counterpart of forms $\mathcal{E}^k_{cpt}(X, p, q)$ is defined similarly.

Define $d_{p, q}:D'^{k-1}(X, p, q) \rightarrow D'^k(X, p, q)$ by
$$d_{p, q}(a, b)=(\pi_p da, \pi'_q db)$$
where
$$\pi_p:D'^k(X) \rightarrow \bigoplus_{\overset{i_1+j_1=k}{i_1<p}}D'^{(i_1,j_1)}(X), \ \
\pi'_q:D'^{k}(X) \rightarrow \bigoplus_{\overset{i_2+j_2=k}{j_2<q}}D'^{(i_2, j_2)}(X)$$
are the natural projections. It is easy to see that $d^2_{p, q}=0$ and hence $(D^{\bullet}(X, p, q), d_{p, q})$ is a cochain complex.
Let $I^k(X)$ be the space of locally integral currents of degree $k$ on $X$.
Define
$\Psi_{p, q}:I^k(X) \rightarrow D'^{k}(X, p, q)$ by
$$\Psi_{p, q}(r)=(\pi_p(r), \pi'_q(r))$$

\begin{proposition}
Let $X$ be a complex manifold of dimension $m$. For $\alpha\in \mathcal{E}^{p, q}(X), \beta\in \mathcal{E}^{m-p, m-q}_{cpt}(X)$, define
$\alpha(\beta):=\int_X\alpha \wedge \beta$. Then $\mathcal{E}^{p, q}(X)$ may be considered as a subspace of $D'^{p, q}(X)$.
With maps and differentials defined above, the triple
$(D'^{\bullet}(X, p, q), \mathcal{E}^{\bullet}(X, p, q), I^{\bullet}(X))$ forms a spark complex.
\end{proposition}

\begin{proof}
It is well known that the inclusion map $\mathcal{E}^*(X, p, q) \hookrightarrow D'^*(X, p, q)$ is a quasi-isomorphism and $\Psi:I^0(X) \rightarrow D'^0(X, p, q)$ is
injective. The fact that $\mathcal{E}^*(X, p, q)\cap \Psi(I^*(X))=\{0\}$ follows from \cite[Appendix B]{HL1}.
\end{proof}

\begin{definition}
For a complex manifold $X$, the $k$-th Harvey-Lawson spark group of level $(p, q)$ is the spark group
$$\widehat{H}^k(X, p, q):=\widehat{H}^k(D^{\bullet}(X, p, q), \mathcal{E}^{\bullet}(X, p, q), I^{\bullet}(X))$$
\end{definition}

\begin{proposition}\label{conjugation iso}
On a complex manifold $X$, the complex conjugation on currents induced by the complex structure of $X$ induces a map $\widehat{H}^k(X, p, q) \rightarrow \widehat{H}^k(X, q, p)$ defined by
$$(a, b, r) \mapsto (\overline{b}, \overline{a}, r)$$ which is an isomorphism.
\end{proposition}

\begin{proof}
Note that $\overline{\pi_p(a)}=\pi'_p(\overline{a})$ and $\overline{\pi'_q(b)}=\pi_q(\overline{b})$. If $d_{p, q}(a, b)=(e_1, e_2)-(\pi_p(r),
\pi_q'(r))$, then $\overline{d_{p, q}(a, b)}=(\pi'_pd\overline{a}, \pi_qd\overline{b})=(\overline{e}_1, \overline{e}_2)- (\pi'_p(\overline{r}),
\pi_q(\overline{r}))$, so $d_{q, p}(\overline{b}, \overline{a})=(\pi_qd\overline{b}, \pi_p'd\overline{a})=(\overline{e}_2,
\overline{e}_1)-(\pi_q(\overline{r}), \pi_p'(\overline{r}))$ this implies that $[(\overline{b}, \overline{a}, \overline{r})]\in \widehat{H}^k(X, q, p)$.
Since $r$ is a locally integral current, it is real, hence $\overline{r}=r$.
This map is well defined and applies it twice we get the minus identity which shows that it is an isomorphism.
\end{proof}

From the general theory of spark complexes \cite[Prop 1.8]{HL1}, we have the following $3\times 3$ commutative grid.
\begin{proposition}\label{3x3}
There is a 3$\times$3 commutative grid of exact sequences associated to the spark complex $(D'^{\bullet}(X, p, q), \mathcal{E}^{\bullet}(X, p, q), I^{\bullet}(X))$
$$\xymatrix{ & 0 \ar[d] & 0 \ar[d] & 0 \ar[d] &\\
0 \ar[r]& \frac{H^k(D'^{\bullet}(X, p, q))}{H^k_I(D'^{\bullet}(X, p, q))} \ar[r] \ar[d] & \widehat{H}^k_E(X, p, q) \ar[r] \ar[d] & d_{p, q}\mathcal{E}^k(X, p, q) \ar[r] \ar[d] & 0\\
0\ar[r]& H^k(G) \ar[r] \ar[d] & \widehat{H}^k(X, p, q) \ar[r]^{\delta_1} \ar[d]^{\delta_2} & Z^{k+1}_I(X, p, q) \ar[r] \ar[d] & 0\\
0 \ar[r]& Ker^{k+1}((\Psi_{p, q})_*) \ar[r] \ar[d] & H^{k+1}(I^{\bullet}(X)) \ar[r]^{(\Psi_{p, q})_*} \ar[d] & H^{k+1}_I(\mathcal{E}^{\bullet}(X, p, q)) \ar[r] \ar[d] & 0\\
& 0 & 0 & 0 &\\}
$$ where $Z^{k+1}_I(X, p, q)$ consists of pairs $(e_1, e_2)\in \mathcal{E}^{k+1}(X, p, q)$ that are closed under the differential $d_{p, q}$ and have integral periods, i.e., $[(e_1, e_2)]=(\Psi_{p, q})_*(\rho)$ in $H^{k+1}(D'^{\bullet}(X, p, q))$ for some $\rho\in
H^{k+1}(I^{\bullet}(X))$. Furthermore, $\widehat{H}^k_E(X, p, q)= \mbox{ kernel of } \delta_2$, and $G$ is the cone complex formed by $\Psi_{p, q}:I^{\bullet}(X) \rightarrow D'^{\bullet}(X, p, q)$.
\end{proposition}

\subsection{Aeppli-Bott-Chern cohomology as a hypercohomology}
Fix a complex manifold $X$. Let $\Omega^k, \overline{\Omega}^k$ be the sheaves of holomorphic $k$-forms and anti-holomorphic $k$-forms on $X$ respectively.
Recall that the Aeppli and Bott-Chern cohomology for a complex manifold $X$ can be defined by the hypercohomology of the complex of sheaves: if $q\geq p$,
$$B^{\bullet}_{p, q}:0 \rightarrow \C \rightarrow \mathcal{O}\oplus \overline{\mathcal{O}} \rightarrow \Omega^1\oplus \overline{\Omega}^1
\rightarrow \cdots \rightarrow \Omega^{p-1}\oplus \overline{\Omega}^{p-1}\rightarrow \overline{\Omega}^p \rightarrow \cdots \rightarrow
\overline{\Omega}^{q-1} \rightarrow 0,$$ we have
$$H^{p, q}_A(X; \C)\cong \H^{p+q+1}(X, B^{\bullet}_{p+1, q+1}) \mbox{ and }
H^{p, q}_{BC}(X; \C)\cong \H^{p+q}(X, B^{\bullet}_{p, q}).$$

Modifying accordingly we have the case for $p\geq q$.

\begin{definition}
Let $\Omega^{\bullet<p, \bullet<q}$ be the complex of sheaves
$$\mathcal{O}\oplus \overline{\mathcal{O}} \rightarrow \Omega^1\oplus \overline{\Omega}^1
\rightarrow \cdots \rightarrow \Omega^{p-1}\oplus \overline{\Omega}^{p-1}\rightarrow \overline{\Omega}^p \rightarrow \cdots \rightarrow
\overline{\Omega}^{q-1} \rightarrow 0$$ if $p<q$, and the complex of sheaves:
$$\mathcal{O}\oplus \overline{\mathcal{O}} \rightarrow \Omega^1\oplus \overline{\Omega}^1
\rightarrow \cdots \rightarrow \Omega^{q-1}\oplus \overline{\Omega}^{q-1}\rightarrow {\Omega}^q \rightarrow \cdots \rightarrow
{\Omega}^{p-1} \rightarrow 0$$ if $p\geq q$
\end{definition}

Similar to the definition of Deligne cohomology, we define Aeppli-Bott-Chern cohomology as following.

\begin{definition}
The Aeppli-Bott-Chern cohomology $H^k_{ABC}(X; \Z(p, q))$ is defined to be the hypercohomology group $\H^k(X, \Z\rightarrow \Omega^{\bullet<p, \bullet<q})$.
If without confusion, we will just call this cohomology the ABC cohomology.
\end{definition}

\begin{proposition}
There is an isomorphism
$$H^k_{ABC}(X; \Z(p, q))\cong H^{k-1}(Cone(I^{\bullet}(X)\overset{\Psi_{p, q}}{\rightarrow} D'^{\bullet}(X, p, q)))$$ where $Cone(I^{\bullet}(X)\overset{\Psi_{p, q}}{\rightarrow} D'^{\bullet}(X, p, q))$ is the cone complex associated to the cochain morphism $\Psi_{p, q}:I^{\bullet}(X) \rightarrow D'^{\bullet}(X, p, q)$.
\end{proposition}

\begin{proof}
We prove only the case $q\geq p$. There are acyclic resolutions
$$\Z \rightarrow I^{\bullet} \mbox{ and } \Omega^{n_i}\oplus \overline{\Omega}^{m_i} \rightarrow D'^{n_i, \bullet}\oplus D'^{\bullet, m_i}$$
Define $\eta_k:I^k \rightarrow D'^{k, 0}\oplus D'^{0, k}$ by
$$\eta_k(r)=(\Pi_{k, 0}(r), \Pi_{0, k}(r))$$ where
$\Pi_{i, j}:I^k \rightarrow D'^{i, j}$ is the natural projection induced from the decomposition
$$I^k \hookrightarrow D'^k=\bigoplus_{i+j=k}D'^{i, j}$$
Then we have a commutating diagram of sheaves:
$$\begin{array}{ccccccc}
     I^{\bullet} & \overset{\eta_*}{\rightarrow} & D'^{\bullet, 0}\oplus D'^{0, \bullet} & \rightarrow & D'^{\bullet, 1}\oplus D'^{1, \bullet} & \rightarrow \cdots\\
    \uparrow &  & \uparrow & & \uparrow & \\
    \Z & \rightarrow & \Omega^0\oplus \overline{\Omega}^0 & \rightarrow & \Omega^1\oplus \overline{\Omega}^1 & \rightarrow \cdots\\
  \end{array}
$$

Let $D'^{i, j}=0$ if $i$ or $j$ equals to -1. Then we have a more uniform expression of the resolution of sheaves
$$\Omega^{n_i}\oplus \overline{\Omega}^{m_i} \rightarrow D'^{n_i, 0}\oplus D'^{0, m_i} \rightarrow D'^{n_i, 1}\oplus D'^{1, m_i}
\rightarrow \cdots \rightarrow D'^{n_i, j}\oplus D'^{j, m_i}$$ where
$$
n_i=\left\{
        \begin{array}{ll}
          i, & \hbox{ if } i<p \\
          -1, & \hbox{ if } i\geq p
        \end{array}
      \right.
m_i=\left\{
  \begin{array}{ll}
    i, & \hbox{ if } i<q\\
    -1, & \hbox{ if } i\geq p
  \end{array}
\right.$$

Let $F^{i, j}=D'^{n_i, j}\oplus D'^{j, m_i}$, then $F^k:=\bigoplus_{i+j=k}F^{i, j}$ and $F^k(X)=D'^k(X, p, q)$.
By \cite[Proposition A.3]{HL1}, the hypercohomology
$$\H^{k}(X, \Z \rightarrow \Omega^{\bullet<p, \bullet<q}) \cong H^{k-1}(Cone(\Psi_{p, q}:I^{\bullet}(X) \rightarrow F^{\bullet}(X)))= H^{k-1}(Cone(\Psi_{p, q}:I^{\bullet}(X) \rightarrow D'^{\bullet}(X, p, q)))$$
\end{proof}

\begin{corollary}
There is a short exact sequence
$$ 0\rightarrow H^{k+1}_{ABC}(X; \Z(p, q)) \rightarrow \widehat{H}^k(X, p, q) \overset{\delta_1}{\rightarrow} Z^{k+1}_I(X, p, q) \rightarrow 0$$
\end{corollary}

\begin{proof}
Consider the $3\times 3$-grid in Proposition \ref{3x3} associated to the spark complex $\mathscr{S}=(D'^{\bullet}(X, p, q), \mathcal{E}^{\bullet}(X, p, q), I^{\bullet}(X))$.
By result above, we may replace the cohomology of the cone complex in the middle row of the 3x3-grid by the ABC cohomology.
\end{proof}

\begin{corollary}
On a complex manifold $X$, the complex conjugation on currents induces an isomorphism between $H^k_{ABC}(X; \Z(p, q))$ and $H^k_{ABC}(X; \Z(q, p))$.
\end{corollary}

\begin{proof}
This follows from Proposition \ref{conjugation iso} by considering $H^k_{ABC}(X; \Z(p, q))$ as a subgroup of $\widehat{H}^{k-1}(X, p, q)$.
\end{proof}

\begin{definition}
On a compact K\"ahler manifold $X$, we define the total Griffiths's $p$-th intermediate Jacobian to be the group
$$ \mathcal{TJ}_p(X):=(F^pH^{2p-1}(X; \C)/H^{2p-1}(X; \Z))\bigoplus (\overline{F^pH^{2p-1}}(X; \C)/H^{2p-1}(X; \Z))$$
where $F^pH^{2p-1}(X; \C)=\underset{\underset{i\geq p}{i+j=2p-1}}\bigoplus H^{i, j}(X)$ is the Hodge filtration and $\overline{F^pH^{2p-1}(X; \C)}$ is the complex conjugation of
$F^pH^{2p-1}(X; \C)$.
\end{definition}

\begin{corollary}
When $p=q, k=2p-1$, on a compact K\"ahler manifold $X$, the $3\times 3$-grid has the form
$$\xymatrix{ & 0 \ar[d] & 0 \ar[d] & 0 \ar[d] &\\
0 \ar[r]& \mathcal{TJ}_p(X) \ar[r] \ar[d] & \widehat{H}^k_E(X, p, p) \ar[r] \ar[d] & d_{p, p}\mathcal{E}^{2p-1}(X, p, p) \ar[r] \ar[d] & 0\\
0\ar[r]& H^{2p}_{ABC}(X; \Z(p, p)) \ar[r] \ar[d] & \widehat{H}^{2p-1}(X, p, p) \ar[r]^{\delta_1} \ar[d]^{\delta_2} & Z^{2p}_I(X, p, p) \ar[r] \ar[d] & 0\\
0 \ar[r]& Hdg^{p, p}(X) \ar[r] \ar[d] & H^{2p}(X; \Z) \ar[r]^{(\Psi_{p, q})_*} \ar[d] & H^{2p}_I(X, p, p) \ar[r] \ar[d] & 0\\
& 0 & 0 & 0 &\\}
$$
where $Hdg^{p, p}(X)$ is the group of Hodge classes.
\end{corollary}

Let $X$ be a complex manifold. Recall that (see \cite{HL1, H}) the Harvey-Lawson spark groups of level $p$ are the spark groups of the spark complex
$$(D'^{\bullet}(X, p), \mathcal{E}^{\bullet}(X, p), I^{\bullet}(X))$$
where
$D'^k(X, p)=\bigoplus_{\overset{i+j=k}{i<p}}D'^{i, j}(X)$,
$\mathcal{E}^k(X, p)=\bigoplus_{\overset{i+j=k}{i<p}}\mathcal{E}^{i, j}(X)$, and $I^{\bullet}(X) \rightarrow D'^{\bullet}(X, p)$ is the projection map.
The Deligne cohomology group $H^{k+1}_{\mathscr{D}}(X;\Z(p))$ sits in the short exact sequence
$$0 \rightarrow H^{k+1}_{\mathscr{D}}(X;\Z(p)) \rightarrow \widehat{H}^{k}(X, p)\overset{\delta_1}{\rightarrow} Z^{k+1}_I(X, p)\rightarrow 0$$

\begin{proposition}
\begin{enumerate}
\item
We have a morphism between spark complexes
$$\xymatrix{ I^{\bullet} \ar[r] \ar@{=}[d] & D'^{\bullet}(X, p, q) \ar@<-6ex>[d] \supseteq  \mathcal{E}^{\bullet}(X, p, q)\ar@<6ex>[d]\\
I^{\bullet}\ar[r] & D'^{\bullet}(X, p) \supseteq \mathcal{E}^{\bullet}(X, p)}$$ \linebreak[2]
where the middle map is given by the natural projection.
This morphism induces a morphism between short exact sequences:
$$\xymatrix{ 0 \ar[r] & H^k_{ABC}(X; \Z(p, q)) \ar[r] \ar[d] & \widehat{H}^{k-1}(X, p, q) \ar[r] \ar[d] & Z^k_I(X, p, q) \ar[r] \ar[d] & 0\\
0 \ar[r] & H^k_{\mathscr{D}}(X; \Z(p)) \ar[r] & \widehat{H}^{k-1}(X, p) \ar[r] & Z^k_I(X, p) \ar[r] & 0}$$

\item For $X$ a complex manifold, there is a commutative diagram
$$\xymatrix{ \widehat{H}^{k}(X, p, q) \ar[r] \ar[d] &  \widehat{H}^{k}(X, q) \ar[d]\\
  \widehat{H}^{k}(X, p) \ar[r] & H^{k}(X; \Z)}$$
given by natural projections which induces a commutative diagram
$$\xymatrix{ H^k_{ABC}(X; \Z(p, q)) \ar[r] \ar[d] &  H^k_{\mathscr{D}}(X; \Z(q)) \ar[d]\\
  H^{k}_{\mathscr{D}}(X, \Z(p)) \ar[r] & H^{k}(X; \Z)}$$

\item For $X$ a compact K\"ahler manifold, $k=p+q-1$, if $H^{k+1}(X; \Z)$ is a free abelian group, then
$$\widehat{H}^k(X, p, q)\cong (\C/ \Z)^t\oplus H^{k+1}(X; \Z)\oplus d_{p, q}\mathcal{E}^{k}(X, p, q)$$ where $t=\mbox{dim}_{\C} H^{k}(X; \C)$.

\item For $X$ a complex manifold, there is a commutative diagram:
$$\xymatrix{\widehat{H}^k(X, p+1, q+1) \ar[r]^{\delta_1} \ar[d]_{\delta_2} & Z^{k+1}_I(X, p+1, q+1) \ar[d]\\
H^{k+1}(X; \Z) \ar[r]& \underset{\underset{i<p+1}{i+j=k}}{\bigoplus}H^{i, j}_{A, I}(X)\bigoplus \underset{\underset{j<q+1}{i+j=k}}{\bigoplus}H^{i, j}_{A, I}(X)}$$
where the right vertical arrow is given by $(e_1, e_2)\mapsto ([e_1], [e_2])$, the bottom horizontal arrow is induced by the projection $\Pi_{i, j}:I^{k+1}(X) \rightarrow D'^{i, j}(X)$, and $H^{i, j}_{A, I}(X)$
is the image of the homomorphism $\Pi_{(i, j)*}: H^{k+1}(X; \Z) \rightarrow H^{i, j}_{A}(X)$ where $H^{i, j}_A(X)$ is the $(i, j)$ Aeppli cohomology of $X$.

\end{enumerate}
\end{proposition}

\begin{proof}
\begin{enumerate}
\item This follows directly from definition.

\item The morphisms are
$$\xymatrix{ [(a, b, r)] \ar[r] \ar[d] & [(b, r)]\ar[d]\\
[(a, r)] \ar[r] & [r]}$$

\item In a compact K\"ahler manifold, $k=p+q-1$, $\H^{k}(X; \Omega^{\bullet<p, \bullet<q})=\underset{\underset{r<p}{r+s=k}}{\bigoplus}H^{r, s}(X)\bigoplus \underset{\underset{s<q}{r+s=k}}{\bigoplus}H^{r, s}(X)=H^{k}(X; \C)$. Note that
$H^{k}(\mathcal{E}^{\bullet}(X, p, q))\cong \H^{k}(X;  \Omega^{\bullet<p, \bullet<q})$. Now consider the $3\times 3$-grid associated to the
spark complex $\mathscr{S}=(D'^{\bullet}(X, p, q), \mathcal{E}^{\bullet}(X, p, q), I^{\bullet}(X))$. Since $H^{k+1}(X; \Z)$ is a free abelian group, the middle
column of the $3\times 3$-grid splits. Since $d_{p, q}\mathcal{E}^k(X, p, q)$ is a vector space, the top row of the $3\times 3$-grid also splits. Thus we have
$\widehat{H}^{k}(X, p, q)\cong \frac{H^{k}(\mathcal{E}^{\bullet}(X, p, q))}{H^{k}_I(\mathcal{E}^{\bullet}(X, p, q))}\oplus H^{k+1}(X; \Z)\oplus d_{p, q}\mathcal{E}^k(X, p, q)$ and the result follows.

\item Recall that the Aeppli cohomology is defined as $H^{i, j}_A(X)=\frac{Ker\partial{\overline{\partial}}}{Im\partial+Im\overline{\partial}}$. For $(e_1, e_2)\in Z^k_I(X, p+1, q+1)$, $\pi_{p+1}de_1=0$. By comparing the types of both sides, we get $(\overline{\partial} e^{p, q}_1+\partial e^{p-1, q+1}_1)+\cdots =0$.
This implies that $\overline{\partial}\partial e^{p-i, q+i}_1=0$ for $i=1, 2, ..., p$. Similarly, $\overline{\partial}\partial e^{p+j, q-j}_2=0$ for $j=1, 2, ..., q$. So
$([e_1], [e_2])\in \underset{\underset{i<p+1}{i+j=k}}{\bigoplus}H^{i, j}_{A}(X)\bigoplus \underset{\underset{j<q+1}{i+j=k}}{\bigoplus}H^{i, j}_{A}(X)$.
Note that if $d\alpha=0$, then $\partial\overline{\partial}\alpha^{i, j}=0$ where $\alpha=\sum_{i+j=k+1}\alpha^{i, j}$, and $\Pi_{i, j}(d\beta)=\partial\beta^{i-1, j}+\overline{\partial}\beta^{i, j-1}$ for $\beta=\sum_{i+j=k}\beta^{i, j}$. This implies that $\Pi_{(p, q)*}$ is well defined. The commutativity of this diagram is clear.
Since $\delta_1$ is surjective, the right vertical homomorphism has image as indicated.
\end{enumerate}
\end{proof}

\subsection{Lefschetz property}
Let $X$ be a K\"ahler manifold with K\"ahler form $\omega$. The Lefschetz operator $\L:D'^{\bullet}(X) \rightarrow D'^{\bullet+2}(X)$ is defined
by $\L(\alpha)=\omega\wedge \alpha$. Let us recall that when in addition $X$ is compact, the Lefschetz decomposition of forms induces a decomposition on currents. We summarize several
properties that we need in the following: suppose that the dimension of $X$ is $n$.
\begin{enumerate}
\item $D'^k(X)=\sum_{i\geq i_0}\L^iP^{k-2i}(X)$ where $P^k(X)=\{\alpha\in D'^k(X)|\L^{n-k+1}\alpha=0\}$ is the primitive part, $i_0=max\{i-n, 0\}$,
the Lefschetz operator $\L^{n-k}:D'^k(X) \rightarrow D'^{2n-k}(X)$ is an isomorphism, and $\L^j:D'^i(X) \rightarrow D'^{i+2j}(X)$ is injective
if $j\leq n-i$.
\item If $a=\sum_{i\geq i_0}\L^ia_i\in D'^k(X)$ is the Lefschetz decomposition of $a$ where $i_0=max\{i-n, 0\}$, $a_i\in P^{k-2i}(X)$, define
$Ta=\sum_{i\geq i_1}\L^{i-1}a_i$ where $i_1=max\{i-n, 1\}$, then $T^{n-k}$ is the inverse of $\L^{n-k}:D'^k(X) \rightarrow D'^{2n-k}(X)$ and
$T^{n-k}\circ \L^{n-k}=id_{k-1}:D'^k(X) \rightarrow D'^k(X)$ if $k\leq n$.
\end{enumerate}

\begin{proposition}
Suppose that $p+q=k-1$ and $k\leq n$, then the map $\L^{n-k}$ induces monomorphisms
$$\L^{n-k}:\widehat{H}^{k-1}(X, p, q; \Q)\rightarrow \widehat{H}^{2n-k+1}(X, n-q, n-p; \Q)$$ and $$\L^{n-k}:H^k_{ABC}(X; \Z(p, q); \Q) \rightarrow H^{2n-k}_{ABC}(X; \Z(n-q, n-p); \Q)$$
where $\Q$ indicates the original groups tensored with $\Q$ over $\Q$. Furthermore, these monomorphisms are isomorphisms if the primitive
cohomology $PH^{k-1}(X; \Q)=0$.
\end{proposition}

\begin{proof}
Note that $\L d=d\L$, $Td=dT$ and $\L^i\pi_p=\pi_{p+i}\L^i$, $\L^i\pi'_q=\pi'_{q+i}\L^i$. The maps are well defined and injective
by the properties of Lefschetz decomposition mentioned above. Note that for $[(a', b', r')]\in \widehat{H}^{2n-k+1}(X, n-q, n-p; \Q)$, we have
$[(T^{n-k}a', T^{n-k}b', T^{n-k}r')]\in \widehat{H}^{k-1}(X, p, q; \Q)$, and $\L^{n-k}[(T^{n-k}a', T^{n-k}b', T^{n-k}r')]=[(a'-a_{n-k-1}, b'-b_{n-k-1}, r')]$ where $a'=\sum_{i\geq n-k-1}\L^ia_i, b'=\sum_{i\geq n-k-1}\L^ib_i$ are the Lefschetz decomposition of $a'$ and $b'$. Thus if $PH^{k-1}(X; \Q)=0$,
then $a_{n-k-1}=dc, b_{n-k-1}=de$, and $[(a_{n-k-1}, b_{n-k-1}, 0)]=0$ in $\widehat{H}^{k-1}(X, p, q; \Q)$. By
restriction, the same holds for Aeppli-Bott-Chern cohomology with $\Q$-coefficients.
\end{proof}

\section{Ring structure on $H^{\bullet}_{ABC}(X; \Z(*, *))$}
Let $X$ be a complex manifold of complex dimension $m$.

\begin{definition}
Let $(D'^k(X))^2=D'^k(X)\oplus D'^k(X), (\mathcal{E}^k(X))^2=\mathcal{E}^k(X)\oplus \mathcal{E}^k(X)$, and $\Psi:I^k(X) \rightarrow (D'^k(X))^2$ be defined by
$r \mapsto (r, r)$. Then $((D'^{\bullet}(X))^2, (\mathcal{E}^{\bullet}(X))^2, I^{\bullet}(X))$ is a spark complex. Let
$$\widehat{H}^k_{D^2}(X):=\widehat{H}^k((D'^{\bullet}(X))^2, (\mathcal{E}^{\bullet}(X))^2, I^{\bullet}(X))$$
\end{definition}

To define a ring structure on $\widehat{H}^{\bullet}_{D^2}(X)$, we need a modified version of \cite[Thm D.1]{HLZ}.
If $(a, r)$ is a spark and $da=e-r$, we write $d_1a=e, d_2a=r$.

\begin{lemma}\label{intersect properly}
For given $\alpha\in \widehat{H}^k_{D^2}(X), \beta\in \widehat{H}^{\ell}_{D^2}(X)$ with $k+\ell\leq 2m$ and $(a_1, a_2, r)\in \alpha$, there is
representative $(b'_1, b'_2, s')\in \beta$ such that if $d(a_1, a_2)=(e_1, e_2)-(r, r)$, $d(b'_1, b'_2)=(\widetilde{e}_1, \widetilde{e}_2)-(s, s)$, then
$a_1\wedge b'_1, a_1\wedge s', r\wedge b'_1, r\wedge s, a_2\wedge b'_2, a_2\wedge s', r\wedge b'_2$ are well defined and $r\wedge s'$ is rectifiable.
\end{lemma}

\begin{proof}
Let us recall the construction in \cite[Thm D.1]{HLZ}. For $[(a, R)]\in \widehat{H}^k(X), [(b, S)]\in \widehat{H}^{\ell}(X)$ with $k+\ell\leq 2m$, $db=\psi-S$,
there
is a current $b':=f_{\xi*}b+\chi+\eta$ where $\chi$ is a smooth $\ell$-form, $\eta$ is a smooth $d$-closed $\ell$-form,
for which $a\wedge b', a\wedge d_2b',  R\wedge b'$ and $R\wedge d_2b'$ are well defined, the last one is
rectifiable and $(b', f_{\xi*}S)$ is equivalent to $(b, S)$. The functions $f_{\xi}:X \rightarrow X$ are diffeomorphisms close to identity parametrized
by points $\xi\in \R^N$ for some $N$. Note that $db'=\psi-f_{\xi*}S$ and $d_2b'=f_{\xi*}S$. Now we fix two representatives $(a_1, a_2, r)\in \alpha, (b_1,
b_2, s)\in \beta$. Since $[(a_1, r)], [(a_2, r)]\in \widehat{H}^k(X), [(b_1, s)], [(b_2, s)]\in \widehat{H}^{\ell}(X)$, by the construction above, we may
choose $\xi\in \R^N$ such that $a_1\wedge f_{\xi*}b_1, a_1\wedge f_{\xi*}s, r\wedge f_{\xi*}b_1, r\wedge f_{\xi*}s, a_2\wedge
f_{\xi*}b_2, a_2\wedge f_{\xi*}s, r\wedge f_{\xi*}b_2$ are all simultaneously well defined and $r\wedge f_{\xi*}s$ is rectifiable.

As in the Harvey-Lawson-Zweck's construction, there exist some smooth forms $\chi_1, \eta_1, \chi_2, \eta_2$ and
$$b_1':=f_{\xi*}b_1+\chi_1+\eta_1, \ \ \ b_2':=f_{\xi*}b_2+\chi_2+\eta_2$$
such that $(b_1', f_{\xi*}s)$ and $(b'_2, f_{\xi*}s)$ are equivalent to $(b_1, s)$ and $(b_2, s)$ respectively in $\widehat{H}^{\ell}(X)$. So by definition, $(b_1', b_2',
f_{\xi*}s)\in \beta$ and the products mentioned in the statement of the Lemma are well defined and $r\wedge f_{\xi*}s$ is rectifiable.
\end{proof}

If $da=\phi-R, db=\psi-S$ and the product is well defined for these two sparks, we write
$$a*b:=a\wedge \psi+(-1)^{k+1}R\wedge b$$
We denote by $\sim$ for the equivalence of two sparks.

\begin{lemma}
If $(a_1, a_2, r)\sim (a'_1, a'_2, r')$, $(b_1, b_2, s)\sim (b'_1, b'_2, s')$ are sparks of the spark complex $((D'^{\bullet}(X))^2, (\mathcal{E}^{\bullet}(X))^2, I^{\bullet}(X))$ and the equivalences are given by
$$\left\{
    \begin{array}{ll}
      a'_1-a_1=d\widetilde{u}_1+\widetilde{R}, \\
      a'_2-a_2=d\widetilde{u}_2+\widetilde{R}, \\
      r'-r=-d\widetilde{R}
    \end{array}
  \right.,
\left\{
  \begin{array}{ll}
    b'_1-b_1=d\widetilde{v}_1+\widetilde{T},\\
    b'_2-b_2=d\widetilde{v}_2+\widetilde{T},\\
    s'-s=-d\widetilde{T}
  \end{array}
\right.
$$
respectively. Then there exist
\begin{enumerate}
\item integral current $R=\widetilde{R}+d\sigma_1$ such that $R\wedge s'$ is well defined and rectifiable;
\item current $u_1=\widetilde{u}_1-\sigma_1+d\sigma_2$ such that $u_1\wedge s', du_1\wedge s'$ are well defined;
\item current $u_2=\widetilde{u}_2-\sigma_1+d\sigma_3$ such that $u_2\wedge s', du_2\wedge s'$ are well defined;
\item integral current $T=\widetilde{T}+d\sigma_4$ such that $a_1\wedge T$, $a_1\wedge dT$, $a_2\wedge T$, $a_2\wedge dT$, $T\wedge r$ are well defined and $T\wedge r$ is rectifiable.
\end{enumerate}
for some currents $\sigma_1, \sigma_2, \sigma_3$ and $\sigma_4$.
Furthermore, we may rewrite the equivalences of sparks as following:
$$\left\{
    \begin{array}{ll}
      a'_1-a_1=du_1+R, \\
      a'_2-a_2=du_2+R, \\
      r'-r=-dR
    \end{array}
  \right.,
\left\{
  \begin{array}{ll}
    b'_1-b_1=dv_1+T,\\
    b'_2-b_2=dv_2+T,\\
    s'-s=-dT
  \end{array}
\right.
$$
where $v_1=\widetilde{v}_1-\sigma_4, v_2=\widetilde{v}_2-\sigma_4$.
\end{lemma}

\begin{proof}
This follows from Federer's slicing theory by making a small perturbation of $\widetilde{R}, \widetilde{u}_1, \widetilde{u}_2, \widetilde{T}$
respectively (see \cite[Theorem A.2]{HLZ}).
\end{proof}

\begin{proposition}
Suppose that $(a_1, a_2, r)\sim (a'_1, a'_2, r'), (b_1, b_2, s)\sim (b'_1, b'_2, s')$ are equivalent sparks, $(a_1, a_2, r)$ meets $(b_1, b_2, s)$ and $(a'_1, a'_2, r')$ meets
$(b'_1, b'_2, s')$ properly respectively. Then $(a'_1*b'_1, a'_2*b'_2, r'\wedge s')\sim (a_1*b_1, a_2*b_2, r\wedge s)$.
\end{proposition}

\begin{proof}
By Lemma above, we may assume that
$$\left\{
    \begin{array}{ll}
      a'_1-a_1=du_1+R, \\
      a'_2-a_2=du_2+R, \\
      r'-r=-dR
    \end{array}
  \right.,
\left\{
  \begin{array}{ll}
    b'_1-b_1=dv_1+T,\\
    b'_2-b_2=dv_2+T,\\
    s'-s=-dT
  \end{array}
\right.
$$
where $a_1\wedge T, a_2\wedge T, a_1\wedge dT, a_2\wedge dT, R\wedge s', T\wedge r$ are well defined and the last two currents are rectifiable.
Suppose that
$$\left\{
    \begin{array}{ll}
      da_1 =e_1-r, \\
      da'_1=e_1-r',\\
      da_2 =e_2-r,\\
      da'_2=e_2-r'
    \end{array}
  \right.,
\left\{
  \begin{array}{ll}
    db_1 =f_1-s,\\
    db'_1=f_1-s',\\
    db_2 =f_2-s,\\
    db'_2=f_2-s'
  \end{array}
\right.
$$
Then
\begin{align*}
a'_1*b'_1-a_1*b_1=& a'_1\wedge s'+(-1)^{k+1}e_1\wedge b'_1-a_1\wedge s-(-1)^{k+1}e_1\wedge b_1\\
                 =& (a_1+du_1+R)\wedge s'-a_1\wedge s+(-1)^{k+1}e_1\wedge (dv_1+T)\\
                 =& a_1\wedge (-dT)+du_1\wedge s'+R\wedge s'+d(e_1\wedge v_1)+(-1)^{k+1}e_1\wedge T\\
                 =& (-1)^{k+1}da_1\wedge T-a_1\wedge dT+d(u_1\wedge s'+e_1\wedge v_1)+(-1)^{k+1}r\wedge T+R\wedge s'\\
                 =& d((-1)^{k+1}a_1\wedge T+u_1\wedge s'+e_1\wedge v_1)+(-1)^{k+1}r\wedge T+R\wedge s'
\end{align*}
Similarly, $$a'_2*b'_2-a_2*b_2=d((-1)^{k+1}a_2\wedge T+u_2\wedge s'+e_2\wedge v_2)+(-1)^{k+1}r\wedge T+R\wedge s'$$
and we have $$d((-1)^{k+1}r\wedge T+R\wedge s')=r\wedge s-r'\wedge s'$$
This completes the proof.
\end{proof}

\begin{definition}
Suppose that $\alpha\in \widehat{H}^k_{D^2}(X), \beta\in \widehat{H}^{\ell}_{D^2}(X)$ with $k+\ell\leq 2m$.
For any representative $(a_1, a_2, r)\in \alpha$, choose representative $(b_1', b_2', s')\in \beta$ according to Lemma \ref{intersect properly}, we define
$$\alpha*\beta:=[(a_1*b_1', a_2*b_2', r\wedge s')]\in \widehat{H}^{k+\ell+1}_{D^2}(X)$$
\end{definition}

By Proposition above, this product is well defined. A direct computation shows that it is graded-commutative. This gives us the following result.

\begin{proposition}
$\widehat{H}^{\bullet}_{D^2}(X)$ is a graded-commutative ring.
\end{proposition}

To define a product on $H^{\bullet}_{ABC}(X; \Z(p, q))$, we first define a product on $\widehat{H}^{\bullet}(X, p, q)$ and then reduce the product to $H^{\bullet}_{ABC}(X; \Z(p, q))$.
Note that $d_p\pi_p=\pi_pd$ and $d_q\pi'_q=\pi'_qd$.

Let
$(\pi_p, \pi'_q, id)_k:\widehat{H}^k_{D^2}(X) \rightarrow \widehat{H}^k(X, p, q)$ be the map defined by
$$[(a, b, r)] \mapsto [(\pi_p(a), \pi'_q(b), r)]$$

We first make an observation.

\begin{lemma}
$Ker(\pi_p, \pi'_q, id)_k=\{\alpha\in \widehat{H}^k_{D^2}(X)|\exists (a, b, 0)\in \alpha, a, b \mbox{ smooth }, \pi_p(a)=0, \pi'_q(b)=0\}$.
\end{lemma}

\begin{proof}
Suppose $(a, b, r)\in \alpha\in Ker(\pi_p, \pi'_q, id)_k$. Then there is $(a', b', s)\in D'^{k-1}(X, p, q)\oplus I^k(X)$ such that $(\pi_p(a),
\pi'_q(b))=d_{p, q}(a', b')+(\pi_p(s), \pi'_q(s))$ and $r=-ds$. Let $\widetilde{a}=a-da'-\Psi(s), \widetilde{b}=b-db'-\Psi(s)$ and
$\widetilde{r}=r+ds=0$, then $(a, b)-(\widetilde{a}, \widetilde{b})=d(a', b')+\Psi(s)$. So $(\widetilde{a}, \widetilde{b}, 0)\in \alpha$. The other
direction is clear.
\end{proof}

\begin{theorem}\label{surjective homo}
The map $(\pi_p, \pi'_q, id)_k$ is a surjective group homomorphism and the kernel of the map $(\pi_p, \pi'_q, id)_k$
is an ideal of $\widehat{H}^{\bullet}_{D^2}(X)$.
\end{theorem}

\begin{proof}
Suppose $\alpha\in Ker(\pi_p, \pi'_q, id)_k, \beta\in \widehat{H}^l_{D^2}(X)$, choose representatives $(a, b, 0)\in \alpha$ such that $\pi_p(a)=0,
\pi_q'(b)=0$, and $(a', b', r')\in \beta$ such that the product is well defined. If $D(a', b')=(e_1, e_2)-(r', r')$, then
$$\alpha* \beta=[(a\wedge e_1+(-1)^{k+1}0\wedge r', b*e_2+(-1)^{k+1}0\wedge r', 0\wedge r')]=[(a\wedge e_1, b\wedge e_2, 0)]\in Ker(\pi_p, \pi'_q, id)_k$$
So the kernel is an ideal of $\widehat{H}^{\bullet}_{D^2}(X)$.

To show the surjectivity, we pick $[(a, b, r)]\in \widehat{H}^k(X, p, q)$. Then by definition, $d_{p, q}(a, b)=(e_1, e_2)-\Psi_{p, q}(r)$ and $dr=0$.
From the isomorphism $H^{k+1}(D'^{\bullet}(X)^2)\cong H^{k+1}(\mathcal{E}^{\bullet}(X)^2)$, there is $(a_0, b_0)\in D'^{k+1}(X)^2, (e_0, f_0)\in
\mathcal{E}^{k+1}(X)^2$ such that $d(a_0, b_0)=(e_0, f_0)-(r, r)$. So $d_{p, q}(a_0, b_0)=(\pi_p, \pi'_q)(e_0, f_0)-\Psi_{p, q}(r)$ and this implies
$d_{p, q}((a, b)-(a_0, b_0))=(e_1, e_2)-(\pi_p, \pi'_q)(e_0, f_0)$. By \cite[Lemma 1.5]{HL2}, $(a, b)-(a_0, b_0)=(g_1, g_2)+d_{p, q}(h_1, h_2)$ where
$(g_1, g_2)\in \mathcal{E}^k(X, p, q), (h_1, h_2)\in D'^{k-1}(X, p, q)$. Let $(\widetilde{a}, \widetilde{b})=(a_0, b_0)+(g_1, g_2)+d(h_1, h_2)$. Then
$d(\widetilde{a}, \widetilde{b})=(e_0, f_0)+d(g_1, g_2)-\Psi(r)$. This implies that $[(\widetilde{a}, \widetilde{b}, r)]\in \widehat{H}^k_{D^2}(X)$. Note that
$(\pi_p, \pi'_q)(\widetilde{a}, \widetilde{b})=(\pi_p, \pi'_q)(a_0, b_0)+(g_1, g_2)+(\pi_p, \pi'_q)d(h_1, h_2)=(\pi_p, \pi'_q)(a, b)=(a, b)$. This proves the surjectivity.
\end{proof}

\begin{definition}
Fix $p, q$. Let $\Pi_{p, q}=\bigoplus^{2n}_{k=0}(\pi_p, \pi'_q, id)_k$. Then by Theorem \ref{surjective homo}, the kernel of $\Pi_{p, q}$ is an ideal of
$\widehat{H}^{\bullet}_{D^2}(X)$ and $\Pi_{p, q}$ is surjective. So we have a group isomorphism
$$\widehat{H}^{*}(X, p, q)\cong \widehat{H}^{*}_{D^2}(X)/Ker \Pi_{p, q}$$ The right hand side has a natural ring structure and we define the ring structure
of $\widehat{H}^{*}(X, p, q)$ by this isomorphism.
\end{definition}

For $\alpha\in H^{k+1}_{ABC}(X; \Z(p, q))$ and $\beta\in H^{\ell+1}_{ABC}(X; \Z(p, q))$ where $k+\ell\leq 2m$, we consider them as elements in $\widehat{H}^k(X, p, q)$ and $\widehat{H}^{\ell}(X, p, q)$
respectively. A direct computation shows that $\alpha*\beta$ is in $H^{k+\ell+2}_{ABC}(X; \Z(p, q))$. This shows that $H^{\bullet}_{ABC}(X; \Z(p, q))$ inherits a ring structure from $\widehat{H}^{\bullet}(X, p, q)$.

\begin{corollary}
The ring structure on the Harvey-Lawson spark group $\widehat{H}^{\bullet}(X, p, q)$ induces a ring structure on the ABC cohomology $H^{\bullet}_{ABC}(X; \Z(p, q))$.
\end{corollary}

If we consider the collection of all $ABC$ cohomology $\bigoplus_{k, p, q}H^k(X; \Z(p, q))$, there is also a ring structure on it.

\begin{definition}
Define a product $H^k_{ABC}(X; \Z(p, q))\times H^{\ell}_{ABC}(X; \Z(p', q')) \rightarrow H^{k+\ell}_{ABC}(X; \Z(p+p', q+q'))$ by
$$(\alpha, \beta)\mapsto (\pi_{p+p'}, \pi'_{q+q'}, id)(\widetilde{\alpha}*\widetilde{\beta})$$ where $\widetilde{\alpha}\in
\widehat{H}^{k-1}_{D^2}(X), \widetilde{\beta}\in \widehat{H}^{\ell-1}_{D^2}(X)$ are lifts of $\alpha$ and $\beta$ respectively.
\end{definition}

To verify that this product is well defined, we refer the reader to the proof of \cite[Theorem 6.6]{H} where a similar verification for Deligne
cohomology was done. The following result is also clear from the definition.

\begin{corollary}
The natural map $\bigoplus_{k, p, q} H^k_{ABC}(X; \Z(p, q)) \rightarrow \bigoplus_{k, p} H^k_{\mathscr{D}}(X; \Z(p))$ induced from the projection $[(a_1, a_2, r)]\mapsto [(a_1, r)]$
is a ring homomorphism.
\end{corollary}

\section{K-theory and refined Chern classes}
By a result of Cheeger and Simons \cite{CS}, each smooth complex vector bundle with unitary connection $\nabla$ over a smooth manifold $X$ is assigned differential cohomology class  $\widehat{c}_k(E, \nabla)\in \widehat{H}^{2k-1}(X)$ 
for each $k\geq 0$. Such classes are called refined Chern classes as they satisfy $\delta_1(\widehat{c}_k(E, \nabla))=c_k(\Omega^{\nabla})$ the Chern-Weil form, and $\delta_2(\widehat{c}_k(E, \nabla))=c_k(E)$ the Chern class of $E$.
They also proved a Whitney product formula
$$\widehat{c}(E\oplus E', \nabla\oplus \nabla')=\widehat{c}(E, \nabla)*\widehat{c}(E', \nabla')$$
where $\widehat{c}$ is the Cheeger-Simons total refined Chern class. In this section, we are going to define refined Chern classes in ABC cohomology and prove some results analogous to the classical counterparts.
The model we use for $\widehat{H}^{\bullet}(X)$ is the spark group of the spark complex $(D'^{\bullet}(X), \mathcal{E}^{\bullet}(X), I^{\bullet}(X))$.

\begin{definition}
Let $X$ be a complex manifold and $E$ be a complex vector bundles over $X$ with unitary connection $\nabla$.
Suppose that $\widehat{c}_k(E, \nabla)=[(a, r)]\in\widehat{H}^{2k-1}(X)$. Then $[(a, \overline{a}, r)]\in \widehat{H}^{2k-1}_{D^2}(X)$. We define
$$\widehat{\widehat{c}}_k(E, \nabla)=[(a, \overline{a}, r)]$$ and
$$\widehat{f}_k(E, \nabla):=(\pi_k, \pi'_k, id)(\widehat{\widehat{c}}_k(E, \nabla))\in \widehat{H}^{2k-1}(X, k, k)$$
\end{definition}

We first observe that the product in $\widehat{H}^{\bullet}(X)$ commutes with the complex conjugation.
\begin{lemma}\label{product conjugation}
For $\alpha\in \widehat{H}^k(X), \beta\in \widehat{H}^{\ell}(X)$,
$$\overline{\alpha*\beta}=\overline{\alpha}*\overline{\beta}$$
\end{lemma}

\begin{proof}
Choose representatives $(a, R)\in \alpha, (b, S)\in \beta$ such that the product $(a, R)*(b, S)$ is well defined.
Write $da=\phi-R, db=\psi-S$. Then
$(a, R)*(b, S)=(a\wedge \psi+(-1)^{k+1}R\wedge b, R\wedge S)$
and we have
$\overline{(a, R)}*\overline{(b, S)}=(\overline{a}, R)*(\overline{b}, S)=(\overline{a}\wedge \overline{\psi}+(-1)^{k+1}R\wedge \overline{b}, R\wedge S)=\overline{(a, R)*(b, S)}$.
This gives us the desire formula.
\end{proof}

\begin{theorem}\label{Whitney formula}
Let $E$ and $F$ be two complex vector bundles on a complex manifold $X$ with unitary connections $\nabla$ and $\nabla'$ respectively. There is a Whitney product formula
\begin{enumerate}
\item
$$\widehat{\widehat{c}}(E\oplus F, \nabla\oplus \nabla')=\widehat{\widehat{c}}(E, \nabla)*\widehat{\widehat{c}}(F, \nabla').$$
\item
$$\widehat{f}(E\oplus F, \nabla\oplus \nabla')=\widehat{f}(E, \nabla)*\widehat{f}(F, \nabla')$$
\end{enumerate}
\end{theorem}

\begin{proof}
The first result follows from the Whitney product formula proved by Cheeger and Simons and the Lemma above. For the second result, note that by the definition of the product $*$ of $\bigoplus_{k, p, q}\widehat{H}^k(X, p, q)$ and the result above, we have
\begin{align*}
\widehat{f}_k(E\oplus F, \nabla\oplus \nabla')&=(\pi_k, \pi'_k, id)(\widehat{\widehat{c}}_k(E\oplus F, \nabla\oplus \nabla'))=\sum_{i+j=k}(\pi_k, \pi'_k, id)(\widehat{\widehat{c}}_i(E, \nabla)*\widehat{\widehat{c}}_j(F, \nabla'))\\
&=\sum_{i+j=k}\widehat{f}_i(E, \nabla)*\widehat{f}_j(F, \nabla')
\end{align*}
This gives us the desire formula.
\end{proof}

\begin{remark}
If $E$ is a hermitian bundle and $\nabla$ is the canonical connection associated to the hermitian metric of $E$, then the Chern-Weil form $c_k(\Omega^{\nabla})$
is of type $(k, k)$ and hence $\delta_1(\widehat{f}(E, \nabla))=0$. This implies that $\widehat{f}_k(E, \nabla)\in H^{2k}_{ABC}(X; \Z(k, k))$.
\end{remark}

\begin{proposition}
Let $E$ be a hermitian vector bundle over a complex manifold $X$ and $\nabla$ be the canonical connection associated to the hermitian metrics of $E$.
\begin{enumerate}
\item
The class $\widehat{f}(E, \nabla)\in H^{2k}_{ABC}(X;\Z(k, k))$ is independent of the choice of hermitian metric on $E$.
\item
Under the canonical map from $H^{2k}_{ABC}(X; \Z(k, k)) \rightarrow H^{2k}_{\mathscr{D}}(X;\Z(k))$, the class $\widehat{f}_k$ is sent to $\widehat{d}_k$ where
$\widehat{d}_k$ is the Harvey-Lawson's refined Chern class.
\end{enumerate}
\end{proposition}

\begin{proof}
Suppose that $\widehat{c}_k(E, \nabla_1)=[(a_1, r_1)]\in \widehat{H}^{2k-1}(X), \widehat{c}_k(E, \nabla_2)=[(a_2, r_2)]\in \widehat{H}^{2k-1}(X)$.
By \cite[Proposition 12.1]{HL1}, Harvey and Lawson showed that their refined Chern classes in Deligne cohomology are independent of the choice of hermitian metrics on $E$,
hence $[(\pi_ka_1, r)]=[(\pi_ka_2, r)]\in H^{2k}_{\mathscr{D}}(X; \Z(k))$. This means that there exist $b\in D'^k(X, k)$, $s\in I^k(X)$ such that
$$\left\{
  \begin{array}{ll}
    \pi_k a_1-\pi_k a_2=\pi_kdb+\pi_k(s), \\
    r_1-r_2=-ds,
  \end{array}
\right.$$
Note that $\overline{\pi_k a}=\pi'_k\overline{a}$ and $d$ is a real operator. By taking the complex conjugation of the first equation, we get
$$\pi'_k\overline{a}_1-\pi'_k\overline{a}_2=\pi'_kd\overline{b}+\pi'_k(s)$$
Together with equations above, this means that
$\widehat{f}_k(E, \nabla_1)=[(\pi_ka_1, \pi'_k\overline{a}_1, r_1)]=[(\pi_ka_2, \pi'_k\overline{a}_2, r_2)]=\widehat{f}_k(E, \nabla_2)$.
The class $\widehat{f}_k$ is sent to $\widehat{d}_k$ follows directly from the definition.
\end{proof}

\begin{definition}
If $E$ is a hermitian vector bundle of rank $k$ on a complex manifold $X$, since refined Chern classes of $E$ are independent of hermitian metrics on $E$, we write
$\widehat{f}_k(E)$ for $\widehat{f}_k(E, \nabla)$ where $\nabla$ is the canonical connection associated to a hermitian metric of $E$, and write
the total refined Chern class to be
$$\widehat{f}(E):=1+\widehat{f}_1(E)+\cdots +\widehat{f}_k(E)\in \bigoplus^k_{i=0}H^{2i}_{ABC}(X; \Z(i, i))$$
\end{definition}

\begin{theorem}
For any short exact sequence
$$0 \rightarrow E_1 \rightarrow E_2 \rightarrow E_3 \rightarrow 0$$
of holomorphic vector bundles over $X$, we have
$$\widehat{f}(E_2)=\widehat{f}(E_1)* \widehat{f}(E_3)$$
\end{theorem}

\begin{proof}
Similar to the proof of Theorem \ref{Whitney formula}.
\end{proof}

\begin{corollary}
If $X$ is a complex manifold and $K_{hol}(X)$ is the Grothendieck group of holomorphic vector bundles on $X$, then the total refined Chern class defines a natural map
$$\widehat{f}:K_{hol}(X) \rightarrow \bigoplus_{i\geq 0}H^{2i}_{ABC}(X; \Z(i, i))$$
\end{corollary}

\section{ABC-cohomology of the Iwasawa manifold and its small deformations}
In this section, we compute the ABC-cohomology of the Iwasawa manifold and its small deformations. The Dolbeault cohomology of the Iwasawa manifold and its small deformations were computed by Nakamura in \cite{Nak} and the Bott-Chern and Aeppli cohomology were computed by Angella in \cite{A}. We use an expression of a system of local holomorphic coordinates given in \cite{A} and recall some results that are used in our computation.

Let
$$\H(3; \C):=\left \{ \left(
                 \begin{array}{ccc}
                   1 & z^1 & z^3 \\
                   0 & 1 & z^2 \\
                   0 & 0 & 1 \\
                 \end{array}
               \right):z^1, z^2, z^3\in \C \right\}\subset\mbox{GL}(3; \C)$$
be the 3-dimensional Heisenberg group over $\C$ and consider the action on the left of $\H(3; \Z[i]):=\H(3; \C)\cap \mbox{GL}(3; \Z[i])$ on
$\H(3; \C)$. The compact quotient
$$\I_3:=\H(3; \Z[i])\backslash \H(3; \C)$$ is call the Iwasawa manifold whose $\H(3; \C)$-left-invariant complex structure $J_0$ is the one
inherited by the standard complex structure on $\C^3$.

We recall a theorem of Nakamura \cite{Nak}.

\begin{theorem}
There exists a locally complete complex-analytic family of complex structures $\{X_{\textbf{\mbox{t}}}=(\I_3, J_{\textbf{\mbox{t}}})\}_{\textbf{\mbox{t}}\in \Delta(0, \epsilon)}$, deformations
of $\I_3$, depending on
$$\textbf{\mbox{t}}=(t_{11}, t_{12}, t_{21}, t_{22}, t_{31}, t_{32})\in \Delta(0, \epsilon)\subset \C^6$$
where $\Delta(0, \epsilon)$ is a disc centered at $0\in \C^6$ with a small radius $\epsilon$ and $X_0=\I_3$.
\end{theorem}
There is a set of holomorphic coordinates $\xi^1_{\textbf{\mbox{t}}}, \xi^2_{\textbf{\mbox{t}}}, \xi^3_{\textbf{\mbox{t}}}$ for $X_{\textbf{\mbox{t}}}$ depending on $\textbf{\mbox{t}}$ and the local coordinates of $X_0$. Since we do not need their precise
expressions, we refer the reader to \cite[Theorem 3.1]{A}. Let
$$\varphi^1_{\textbf{\mbox{t}}}:=d\xi^1_{\textbf{\mbox{t}}}, \varphi^2_{\textbf{\mbox{t}}}:=d\xi^2_{\textbf{\mbox{t}}} \mbox{ and }
\varphi^3_{\textbf{\mbox{t}}}:=d\xi^3_{\textbf{\mbox{t}}}-z^1d\xi^2_{\textbf{\mbox{t}}}-(t_{21}\overline{z}^1+t_{22}\overline{z}^2)d\xi^1_{\textbf{\mbox{t}}}$$
Complex numbers $\sigma_{1\overline{1}}, \sigma_{1\overline{2}}, \sigma_{2\overline{1}}, \sigma_{2\overline{2}}$ and $\sigma_{12}$
depending only on $\textbf{\mbox{t}}$ are defined through the following equation

$$d\varphi^3_{\textbf{\mbox{t}}}=\sigma_{12}\varphi^1_{\textbf{\mbox{t}}}\wedge \varphi^2_{\textbf{\mbox{t}}}+\sigma_{1\overline{1}}\varphi^1_{\textbf{\mbox{t}}}\wedge\overline{\varphi}^1_{\textbf{\mbox{t}}}+
\sigma_{1\overline{2}}\varphi^1_{\textbf{\mbox{t}}}\wedge \overline{\varphi}^2_{\textbf{\mbox{t}}}+
\sigma_{2\overline{1}}\varphi^2_{\textbf{\mbox{t}}}\wedge\overline{\varphi}^1_{\textbf{\mbox{t}}}+\sigma_{2\overline{2}}\varphi^2_{\textbf{\mbox{t}}}
\wedge \overline{\varphi}^2_{\textbf{\mbox{t}}}{\textbf{\mbox{t}}}$$

Let $$D({\textbf{\mbox{t}}}):=\mbox{det}\left(
                                          \begin{array}{cc}
                                            t_{11} & t_{12} \\
                                            t_{21} & t_{22} \\
                                          \end{array}
                                        \right)
$$ and
$$S:=\left(
       \begin{array}{cccc}
         \overline{\sigma_{1\overline{1}}} & \overline{\sigma_{2\overline{2}}} & \overline{\sigma_{1\overline{2}}} & \overline{\sigma_{2\overline{1}}} \\
         \sigma_{1\overline{1}} & \sigma_{2\overline{2}} & \sigma_{1\overline{2}} & \sigma_{2\overline{1}} \\
       \end{array}
     \right)$$

Recall that Nakamura classified the small deformations of $\I_3$ into 3 classes: (i), (ii), (iii), and Angella further subdivided
class (ii) into (ii.a) and (ii.b), class (iii) into (iii.a) and (iii.b) by using the Bott-Chern cohomology of $X_{\textbf{\mbox{t}}}$. The classification is given in the following list.
\begin{description}
\item[class (i)]: $t_{11}=t_{12}=t_{21}=t_{22}=0$;
\item[class (ii)]: $D(\textbf{\mbox{t}})=0$ and $(t_{11}, t_{12}, t_{21}, t_{22})\neq (0, 0, 0, 0)$;
                 \begin{description}
                  \item[subclass (ii.a)]: $D(\textbf{\mbox{t}})=0$ and rk$S$=1;
                  \item[subclass (ii.b)]: $D(\textbf{\mbox{t}})=0$ and rk$S$=2;
                  \end{description}
\item[class (iii)]: $D(\textbf{\mbox{t}})\neq 0$;
\begin{description}
\item[subclass (iii.a)]:$D(\textbf{\mbox{t}})\neq 0$ and rk$S$=1;
\item[subclass (iii.b)]:$D(\textbf{\mbox{t}})\neq 0$ and rk$S$=2;
\end{description}
\end{description}

The set $\{\varphi^1_{\textbf{\mbox{t}}}, \varphi^2_{\textbf{\mbox{t}}}, \varphi^3_{\textbf{\mbox{t}}}\}$
is a co-frame of $(1, 0)$-forms on $X_{\textbf{\mbox{t}}}$. The structure equations for
${\textbf{\mbox{t}}}$ in class (i) are
$$
\left\{
  \begin{array}{ll}
    d\varphi^1_{\textbf{\mbox{t}}}=0 \\
    d\varphi^2_{\textbf{\mbox{t}}}=0\\
d\varphi^3_{\textbf{\mbox{t}}}=-\varphi^1_{\textbf{\mbox{t}}}\wedge \varphi^2_{\textbf{\mbox{t}}}\\
\end{array}
\right.
$$

The structure equations for $\textbf{\mbox{t}}$ in class $(ii)$ and $(iii)$ are
$$
\left\{
  \begin{array}{ll}
    d\varphi^1_{\textbf{\mbox{t}}}=0 \\
    d\varphi^2_{\textbf{\mbox{t}}}=0\\
    d\varphi^3_{\textbf{\mbox{t}}}=\sigma_{12}\varphi^1_{\textbf{\mbox{t}}}\wedge \varphi^2_{\textbf{\mbox{t}}}+\sigma_{1\overline{1}}\varphi^1_{\textbf{\mbox{t}}}\wedge\overline{\varphi}^1_{\textbf{\mbox{t}}}+
\sigma_{1\overline{2}}\varphi^1_{\textbf{\mbox{t}}}\wedge \overline{\varphi}^2_{\textbf{\mbox{t}}}+
\sigma_{2\overline{1}}\varphi^2_{\textbf{\mbox{t}}}\wedge\overline{\varphi}^1_{\textbf{\mbox{t}}}+
\sigma_{2\overline{2}}\varphi^2_{\textbf{\mbox{t}}}
\wedge \overline{\varphi}^2_{\textbf{\mbox{t}}}{\textbf{\mbox{t}}}
  \end{array}
\right.
$$

The first step towards our computation of the ABC cohomology of $X_{\textbf{\mbox{t}}}$ is to compute the cohomology group $H^k(\mathcal{E}'^{\bullet}(X_{\textbf{\mbox{t}}}, p))$ for all $k, p$
where $\mathcal{E}'^{\bullet}(X_{\textbf{\mbox{t}}}, p)=\pi'_p(\mathcal{E}^{\bullet}(X_\textbf{\mbox{t}}))$.
To do this, we reduce the computation to the corresponding cohomology of its Lie algebra $\mathcal{G}$. Similar reduction for Bott-Chern cohomology is given in \cite[Theorem 3.7]{A}.
The hypothesis of the following result is satisfied by the Iwasawa manifold and its small deformations.

\begin{proposition}
Let $X=\Gamma\backslash G$ be a solvmanifold endowed with a $G$-left-invariant complex
structure $J$, and $\mathcal{G}$ be the Lie algebra naturally associated with $G$. Denote by $\mathcal{E}'^k_{\mathcal{G}}(X, p)=\pi'_p(\mathcal{E}^k_{\mathcal{G}}(X))$ where
$\mathcal{E}^k_{\mathcal{G}}$ is the vector space of all $G$-left-invariant $k$-forms on $X$. If the De Rham cohomology, $\partial$-cohomology and Bott-Chern cohomology of
$X$ can be computed by the complex of $G$-left-invariant forms,
then the inclusion of the subcomplex $i:\mathcal{E}'^{\bullet}_{\mathcal{G}}(X, p)\hookrightarrow \mathcal{E}'^{\bullet}(X, p)$ is a quasi-isomorphism, which means that
the induced homomorphism $$i_*:H^k(\mathcal{E}'^{\bullet}_{\mathcal{G}}(X, p)) \rightarrow H^k(\mathcal{E}'^{\bullet}(X, p))$$ is an isomorphism for all $k, p\in \Z$.
\end{proposition}

\begin{proof}
For $[\alpha]\in H^k(\mathcal{E}'^{\bullet}(X, p))$, write $\alpha=\alpha^{k, 0}+\cdots+\alpha^{k-p+1, p-1}$. Then from $d'_p\alpha=0$, we get a system of equations
$$\left\{
  \begin{array}{ll}
    \partial \alpha^{k, 0}=0 \Rightarrow \overline{\partial}\partial \alpha^{k, 0}=0,\\
    \overline{\partial}\alpha^{k, 0}+\partial \alpha^{k-1, 1}=0 \Rightarrow \overline{\partial}\partial \alpha^{k-1, 1}=0\\
    \vdots \\
    \overline{\partial}\alpha^{k-q+2, q-2}+\partial \alpha^{k-p+1, p-1}=0\Rightarrow \overline{\partial}\partial \alpha^{k-p+1, p-1}=0\\
  \end{array}
\right.$$
Since by our assumption, the $\partial$-cohomology and Bott-Chern cohomology of $X$ can be computed by $G$-left-invariant forms, so there exist $\beta^{i, j}\in \mathcal{E}'^k_{\mathcal{G}}(X, p)$
for $j=0, 1, ..., k-1$ such that $\alpha^{k, 0}-\beta^{k, 0}=\partial \gamma^{k-1, 0}$, $\alpha^{i, j}-\beta^{i, j}=\partial\overline{\partial}\eta^{i-1, j-1}$ for $j=1, ..., k-1$.
Let $\beta=\beta^{k, 0}+\cdots +\beta^{k-p+1, p-1}\in \mathcal{E}'^k_{\mathcal{G}}(X, p)$.
Then $\alpha-\beta=d'_p(\gamma^{k-1, 0}+\overline{\partial}\eta^{k-1, 1}+\cdots +\overline{\partial}\eta^{k-p, p-2})$. This shows that $i_*$ is surjective.
If $[\omega]\in H^k(\mathcal{E}'^{\bullet}_{\mathcal{G}}(X, p))$ and $i_*[\omega]=0$, we have $\omega=d'_p\rho$ for some $\rho\in \mathcal{E}'^{k-1}(X, p)$.
Comparing the degrees of both sides of this equation, we have
$$\left\{
  \begin{array}{ll}
    \omega^{k, 0}+\cdots+\omega^{k-p+2, p-2}=d(\rho^{k-1, 0}+\cdots+\rho^{k-p+1, p-2}), \\
    \omega^{k-p+1, p-1}=\partial \rho^{k-p, p-1}
  \end{array}
\right.$$

By our assumption, the De Rham cohomology and $\partial$-cohomology of $X$ can be computed by $G$-left-invariant forms, there are $G$-left-invariant $\eta$ and $\tau$ such that
$\omega^{k, 0}+\cdots +\omega^{k-p+2, p-2}=d\eta, \omega^{k-p+1, p-1}=\partial \tau$. Then $\omega=d\eta+\partial \tau=d'_p(\eta+\tau)$ and $\eta+\tau\in \mathcal{E}'^{k+1}_{\mathcal{G}}(X, p)$. This shows that the homomorphism $i_*$ is injective.
\end{proof}

The second step is to show that the integral cohomology groups of the Iwasawa manifold is torsion-free. We combine results developed in \cite{CP, LP} for this goal. The main
 tool we use is the following theorem \cite[Theorem 3]{CP}. For $q\in \N$, let $\Z\{q\}=\Z[\frac{1}{2}, \cdots, \frac{1}{q}]$.
 \begin{theorem}
 For any nilmanifold $N$, $H^*(N; \Z\{q\})$ and $H^*(FL(N); \Z\{q\})$ are isomorphic rings
 where $FL(N)$ denotes the formal group Lie algebra of the fundamental group $G:=\pi_1(N)$ and
 $q\geq d(N)$, where $d(N)$ is equal to the finite sum
 $$d(N)=1+|G_1/G_2|+2|G_2/G_3|+3|G_3/G_4|+\cdots+k|G_k/G_{k+1}|+\cdots$$
 and $G_1\supset G_2\supset G_3\supset \cdots$
 is a descending series of $G$ (see \cite[pg 74]{CP}).
 \end{theorem}

If $N$ is the Iwasawa manifold, after some computation, we get $d(N)=1$ and all cohomology groups $H^*(FL(N); \Z)$ are torsion-free. This implies that the integral cohomology groups of Iwasawa
manifold are torsion-free. Since it is diffeomorphic to its small deformations, we have the following result.

\begin{corollary}
All integral cohomology groups of the Iwasawa manifold and its small deformations are torsion-free.
\end{corollary}

\begin{lemma}
Let $X$ be a complex manifold. Suppose that $H^k(X; \Z)$ is torsion-free. Then rk$H^k_I(\mathcal{E}'^{\bullet}(X, p, q))=\mbox{dim}_{\C} (\pi_{p*}, \pi'_{q*})\mathcal{D}$ where $\mathcal{D}:=\{([\alpha], [\alpha])|[\alpha]\in H^k(\mathcal{E}^{\bullet}(X))\}$ is the
diagonal of $H^k(\mathcal{E}^{\bullet}(X))\oplus H^k(\mathcal{E}^{\bullet}(X))$.
\end{lemma}

\begin{proof}
The inclusion $i:\mathcal{E}^k(X) \hookrightarrow D'^k(X)$ is a quasi-isomorphism, and with the inclusion $j:I^k(X) \hookrightarrow D'^k(X)$,
we have a group homomorphism $\ell_*:=i_*^{-1}\circ j_*: H^k(I^{\bullet}(X)) \rightarrow H^k(\mathcal{E}^{\bullet}(X))$.  Let $\{\phi_1, ..., \phi_n\}$ be a basis of $H^k(I^{\bullet}(X))$.
Since $H^k(X; \Z)$ is torsion-free, the map $\ell_*$ is injective, and hence the rank of the image Im$(\ell_*, \ell_*)$ is $n$.
Let $D_{\R}$ be the real vector subspace of $\mathcal{D}$ obtained by taking linear combination of $\{(\ell_*, \ell_*)(\phi_j, \phi_j)|j=1, ..., n\}$ with real coefficients. Then from the fact $\mbox{dim}_{\R}D_{\R}\leq n$ and $\mbox{Im}(\ell_*, \ell_*)\subset D_{\R}$, we get
$n=\mbox{dim}_{\R}D_{\R}=\mbox{dim}_{\C}\mathcal{D}$.

We have the following commutative diagram
$$\xymatrix{H^k(I^{\bullet}(X)) \ar[r]^-{(\ell_*, \ell_*)} \ar[rd]& \mbox{Im}(\ell_*, \ell_*) \ar[d]^{(\pi_{p*}, \pi'_{q*})} & \hooklongrightarrow & \mathcal{D}_{\R} \ar[d]^{(\pi_{p*}, \pi'_{q*})} &\hooklongrightarrow& \mathcal{D}  \ar[d]^{(\pi_{p*}, \pi'_{q*})}\\
                        & H^k_I(\mathcal{E}^{\bullet}(X, p, q))& \hooklongrightarrow  & (\pi_{p*}, \pi'_{q*})(D_{\R})  &\hooklongrightarrow& (\pi_{p*}, \pi'_{q*})(\mathcal{D}) \\
}$$
and rk$H^k_I(\mathcal{E}^{\bullet}(X, p, q))\leq \mbox{dim}_{\R}(\pi_{p*}, \pi'_{q*})(D_{\R})$.
Since $(\pi_{p*}, \pi'_{q*})\{(\ell_*, \ell_*)(\phi_j, \phi_j)|j=1, ..., n\}$ is a generating set over $\Z$ for $H^k_I(\mathcal{E}^{\bullet}(X, p, q))$, over $\R$ for $(\pi_{p*}, \pi'_{q*})D_{\R}$ respectively, we have
rk$H^k_I(\mathcal{E}^{\bullet}(X, p, q))\geq \mbox{dim}_{\R}(\pi_{p*}, \pi'_{q*})(D_{\R})$. Therefore rk$H^k_I(\mathcal{E}^{\bullet}(X, p, q))=\mbox{dim}_{\R}(\pi_{p*}, \pi'_{q*})(D_{\R})=\mbox{dim}_{\C}(\pi_{p*}, \pi'_{q*})(\mathcal{D})$.
\end{proof}

\begin{lemma}\label{ABC}
If $X$ is a complex manifold and $H^k(X; \Z)$ is torsion-free,
we may write
$$H^k_{ABC}(X; \Z(p, q))\cong\Z^A\oplus (\C/\Z)^B\oplus \C^C$$
where $A=\mbox{rk}H^k(X; \Z)-\mbox{rk}H^{k}_I(D'^{\bullet}(X, p, q))$, $B=\mbox{rk}H^{k-1}_I(D'^{\bullet}(X, p, q))$ and $C=\mbox{dim}H^{k-1}(D'^{\bullet}(X, p, q))-\mbox{rk}H^{k-1}_I(D'^{\bullet}(X, p, q))$.
\end{lemma}

\begin{proof}
Note that since $H^k(X; \Z)$ is torsion-free for all $k\geq 0$, by the $3\times 3$-grid (Proposition \ref{3x3}), we have
$$H^k_{ABC}(X; \Z(p, q))\cong \mbox{Ker}^{k-1}(\Psi_{p, q})_*\oplus \frac{H^{k-1}(D'^{\bullet}(X, p, q))}{H^{k-1}_I(D'^{\bullet}(X, p, q))}\cong \Z^A\oplus (\C/\Z)^B\oplus \C^C$$
as required.
\end{proof}

We list the procedure of our computation of $H^k_{ABC}(X_{\textbf{\mbox{t}}}; \Z(p, q))$ in the following where $X_{\textbf{\mbox{t}}}$ is a small deformation of $\I_3$.
\begin{description}
\item[Step 1]: Find a basis $\psi_1, ..., \psi_s$ consisting of left-invariant forms of $H^k(\mathcal{E}^{\bullet}_{\mathcal{G}}(X_{\textbf{\mbox{t}}}))$.
\item[Step 2]: Compute the dimension of the space generated by $(\pi_{p*}, \pi_{q*}')(\psi_j, \psi_j)$ for $j=1, ..., s$. This gives the dimension of $H^k_I(\mathcal{E}^{\bullet}(X_{\textbf{\mbox{t}}}, p, q))$.
\item[Step 3]: Compute the dimension of $H^k(\mathcal{E}'^{\bullet}_{\mathcal{G}}(X_{\textbf{\mbox{t}}}, q))$. This is equal to the dimension of $H^k(\mathcal{E}^{\bullet}(X_{\textbf{\mbox{t}}}, q))$ and we get the dimension of the group
              $H^k(X_{\textbf{\mbox{t}}}, p, q)$.
\item[Step 4]: Calculate the integers $A, B, C$ as given in Lemma \ref{ABC}.
\end{description}

The following table records the complex dimension of $H^k(\mathcal{E}'^{\bullet}(\I_3, p))$.

\begin{center}
$
\begin{array}{|c | c | c| c |c|c|}
\hline
p \backslash k & 1 & 2 & 3 & 4 & 5 \\
\hline
1              & 2 & 2 & 1 & 0  & 0 \\
\hline
2              & 5 & 9 & 8 & 3 & 0\\
\hline
3              & 4 & 8 & 9 & 7 & 3 \\
\hline
4              & 4 & 8 & 10 & 8 & 4  \\
\hline
\end{array}$
\end{center}

The following table records the complex dimension of $H^k(\mathcal{E}^{\bullet}(\I_3, p, q))$ and the rank of $H^k_I(\mathcal{E}^{\bullet}(\I_3, p, q))$.
Note that $H^k(\mathcal{E}^{\bullet}(\I_3, p, q))\cong H^k(\mathcal{E}^{\bullet}(\I_3, p))\oplus H^k(\mathcal{E}'^{\bullet}(\I_3, q))$ and
the complex conjugation induces an isomorphism between $H^k(\mathcal{E}^{\bullet}(\I_3, p))$ and $H^k(\mathcal{E}'^{\bullet}(\I_3, p))$.
Furthermore,
$H^k(\mathcal{E}^{\bullet}(\I_3, p, q))\cong H^k(\mathcal{E}^{\bullet}(\I_3, q, p))$.

\begin{center}
$
\begin{array}{|c | c | c| c |c|c||c|c|c|c|c|}
\hline
&&&H^k(\mathcal{E}^{\bullet}(\I_3, p, q)) &&&  &&& H^k_I(\mathcal{E}^{\bullet}(\I_3, p, q)) &\\
\hline
(p, q) \backslash k & 1 & 2 & 3 & 4 & 5 &     1 & 2 & 3  & 4 & 5\\
\hline
\hline
(1, 1)              & 4 &  4 &  2 & 0  & 0 &  4 & 4 & 2  & 0 & 0\\
\hline
(1, 2)              & 7 & 11 &  9 & 3  & 0 &  4 & 8 & 6  & 2 & 0\\
\hline
(1, 3)              & 6 & 10 & 10 & 7  & 3 &  4 & 8 & 8  & 6 & 2\\
\hline
(1, 4)              & 6  & 10 & 11 & 8 & 4 &  4 & 8 & 10 & 8 & 4\\
\hline
(2, 2)              & 10 & 18 & 16 & 6 & 0 &  4 & 8 & 10 & 4 & 0\\
\hline
(2, 3)              & 9 & 17 & 17 & 10& 3&    4 & 8 & 10 & 8 & 2\\
\hline
(2, 4)              & 9 & 17 & 18& 11 & 4&    4 & 8 & 10 & 8 & 4\\
\hline
(3, 3)              & 8 & 16 & 18 & 14 & 6&   4 & 8 & 10 & 8 & 4\\
\hline
(3, 4)             & 8 & 16 & 19 & 15 & 7&    4 & 8 & 10 & 8 & 4\\
\hline
(4, 4)             & 8 & 16 & 20 & 16 & 8 &   4 & 8 & 10 & 8 & 4\\
\hline
\end{array}
$
\end{center}

In the following table, we compute $H^k_{ABC}(\I_3; \Z(p, q))$. Each triple in the entries denotes $(A, B, C)$ where $A, B, C$ are given in Lemma \ref{ABC}.

\begin{center}
$
\begin{array}{|c | c | c| c |c|c|c|}
\hline
(p, q) \backslash k & 1 & 2     & 3          & 4          & 5         & 6\\
\hline
\hline
(1, 1) & (0, 1, 0) & (4, 4, 0)  & (8, 4, 0)  & (8, 2, 0)  & (4, 0, 0) & (0, 0, 0)\\
\hline
(1, 2) & (0, 1, 0) & (0, 4, 3)  & (4, 8, 3)  & (6, 6, 3)  & (4, 2, 1) & (0, 0, 0)\\
\hline
(1, 3) & (0, 1, 0) & (0, 4, 2)  & (2, 8, 2)  & (2, 8, 2)  & (2, 6, 1) & (0, 2, 1) \\
\hline
(1, 4) & (0, 1, 2) & (0, 4, 2)  & (0, 8, 2)  & (0, 10, 1) & (0, 8, 0) & (0, 4, 0)  \\
\hline
(2, 2) & (0, 1, 0) & (0, 4, 6)  & (0, 8, 10) & (4, 10, 6) & (4, 4, 2) & (0, 0, 0)\\
\hline
(2, 3) & (0, 1, 0) & (0, 4, 5)  & (0, 8, 9)  & (0, 10, 7) & (2, 8, 2) & (0, 2, 1) \\
\hline
(2, 4) & (0, 1, 0) & (0, 4, 5)  & (0, 8, 8)  & (0, 10, 8) & (0, 8, 3) & (0, 4, 0)  \\
\hline
(3, 3) & (0, 1, 0) & (0, 4, 4)  & (0, 8, 8)  & (0, 10, 8) & (0, 8, 6) & (0, 4, 2)   \\
\hline
(3, 4) & (0, 1, 0) & (0, 4, 4)  & (0, 8, 8)  & (0, 10, 9) & (0, 8, 7) & (0, 4, 3)\\
\hline
(4, 4) & (0, 1, 0) & (0, 4, 4)  & (0, 8, 8)  & (0, 10, 10)& (0, 8, 8) & (0, 4, 4)\\
\hline
\end{array}$
\end{center}

Now we turn to a much more involved computation of the ABC cohomology of $X_{\textbf{\mbox{t}}}$.

The following table records the complex dimension of $H^k(\mathcal{E}^{\bullet}(X_{\textbf{\mbox{t}}}, p))$.

\begin{center}
$
\begin{array}{|c | c | c| c |c|c|}
\hline
p \backslash k & 1 & 2 & 3 & 4 & 5 \\
\hline
1              & 2 & 2 & 1 & 0 & 0 \\
\hline
2              &4 &7 &\mbox{if rkT=1}, 6 &\mbox{if rkT=1}, 2 &0 \\
               &  &  &\mbox{if rkT=2}, 5 &\mbox{if rkT=2}, 1 &  \\
\hline
3              & 4& 8&9 &6 & 2 \\
\hline
4              & 4& 8&10 &8 &4  \\
\hline
\end{array}$
\end{center}

Note that $H^k(\mathcal{E}^{\bullet}(X_{\textbf{\mbox{t}}}, p, q))\cong H^k(\mathcal{E}^{\bullet}(X_{\textbf{\mbox{t}}}, p))\oplus H^k(\mathcal{E}'^{\bullet}(X_{\textbf{\mbox{t}}}, q))$ and
the complex conjugation induces an isomorphism between $H^k(\mathcal{E}^{\bullet}(X_{\textbf{\mbox{t}}}, p))$ and $H^k(\mathcal{E}'^{\bullet}(X_{\textbf{\mbox{t}}}, p))$.
Furthermore,
$H^k(\mathcal{E}^{\bullet}(X_{\textbf{\mbox{t}}}, p, q))\cong H^k(\mathcal{E}^{\bullet}(X_{\textbf{\mbox{t}}}, q, p))$.
Let $N(T)$ denote the number of nonzero entries of $T$ where
$T=\left(
     \begin{array}{cc}
       \sigma_{1\overline{1}} & \sigma_{1\overline{2}} \\
       \sigma_{2\overline{1}} & \sigma_{2\overline{2}} \\
     \end{array}
   \right)
$.
\begin{center}
$
\begin{array}{|c | c | c| c |c|c||c|c|c|c|c|}
\hline
&&& H^k(\mathcal{E}^{\bullet}(X_t, p, q))   &&&  &&&  H^k_I(\mathcal{E}^{\bullet}(X_t, p, q))& \\
\hline
(p, q) \backslash k & 1 & 2 & 3 & 4 & 5 & 1 & 2 & 3 & 4 & 5\\
\hline
(1, 1)              & 4 & 4 & 2 & 0 & 0 & 4 & 4 & 2 & 0 & 0\\
\hline
(1, 2) & 6 & 9 &\mbox{if rkT=1}, 7 & \mbox{if rkT=1}, 2 & 0 & 4 & 7 & 6 & \mbox{if rkT=1}, 2 & 0\\
       &   &   &\mbox{if rkT=2}, 6 & \mbox{if rkT=2}, 1 & 0 &   & &     & \mbox{if rkT=2}, 1 &\\
\hline
(1, 3) & 6 & 10 & 10 & 6 & 2  & 4 & 8 & 10 & \mbox{if rkT=1}, 5& 2\\
       &   &    &    &   &    &   &   &    & \mbox{if rkT=2}, 4&\\
\hline
(1, 4) & 6 & 10 & 11 & 8 & 4 & 4  & 8 & 10 & 8 & 4\\
\hline
(2, 2) & 8 & 14  & \mbox{ if rkT=1}, 12 & \mbox{if rkT=1}, 4 & 0   & 4 & 7 & 10 & \mbox{if rkT=1 and N(T)}=1,     3& 0\\
       &   &     & \mbox{ if rkT=2}, 10 & \mbox{if rkT=2}, 2 &     &&&          & \mbox{if rkT=1 and N(T)}\geq 2, 4&\\
       &   &     &                      &                    &     &&&          & \mbox{if rkT=2}, 2&\\
\hline
(2, 3) & 8 & 15  & \mbox{if rkT=1}, 15& \mbox{if rkT=1}, 8   & 2   & 4 & 8 & 10 & \mbox{if rkT=1}, 6& 2\\
       &   &     & \mbox{if rkT=2}, 14& \mbox{if rkT=2}, 7   &     &&&          & \mbox{if rkT=2}, 4&\\
\hline
(2, 4) & 8 & 15  &\mbox{if rkT=1}, 16  & \mbox{if rkT=1},   10   &  4  & 4 & 8 & 10 & 8& 4\\
       &   &     &\mbox{if rkT=2}, 15  & \mbox{if rkT=2},    9  &    &&&&&\\
\hline
(3, 3) & 8 & 16 & 18 & 12 & 4   & 4 & 8 & 10 & \mbox{if rkT=1}, 6& 4\\
       &   &    &    &    &     &   &   &    & \mbox{if rkT=2}, 4&\\
\hline
(3, 4) & 8 & 16 & 19 & 14 & 6   & 4 & 8 & 10 & 8& 4\\
\hline
(4, 4) & 8 & 16 & 20 & 16 & 8   & 4 & 8 & 10 & 8& 4\\
\hline
\end{array}$
\end{center}

In the following table, we compute $H^k_{ABC}(X_{\textbf{\mbox{t}}}; \Z(p, q))$. Each triple in the entries denotes $(A, B, C)$ where $A, B, C$ are defined in Lemma \ref{ABC}.

\begin{center}
$
\begin{array}{|c | c | c| c |c|c|c|}
\hline
(p, q) \backslash k & 1 & 2     & 3         & 4                           & 5         & 6\\
\hline
(1, 1) & (0, 1, 0)  & (4, 4, 0) & (8, 4, 0) & (8, 2, 0)                   & (4, 0, 0) & (0, 0, 0)\\
\hline
(1, 2) & (0, 1, 0) & (1, 4, 2)  & (4, 7, 2) & \mbox{if rkT=1}, (6, 6, 1)  & (4, 2, 0) & (0, 0, 0)\\
       &           &            &           & \mbox{if rkT=2}, (7, 6, 0)  & (4, 1, 0) & (0, 0, 0)\\
\hline
(1, 3) & (0, 1, 0) & (0, 4, 2) & (0, 8, 2)  & \mbox{if rkT=1}, (1, 10, 0) & (2, 7, 1) & (0, 2, 0) \\
       &           &           &            & \mbox{if rkT=2}, (2, 10, 0) & (2, 6, 2) & (0, 2, 0)  \\
\hline
(1, 4) & (0, 1, 2) & (0, 4, 2) & (0, 8, 2)  & (0, 10, 1)                  & (0, 8, 0) & (0, 4, 0)  \\
\hline
(2, 2) & (0, 1, 0) & (4, 4, 4) & (0, 7, 7)  & \mbox{if rkT=1 and N(T)}=1, (5, 10, 2)     & (4, 3, 1) & (0, 0, 0)\\
       &           &           &            & \mbox{if rkT=1 and N(T)}\geq 2, (4, 10, 2) & (4, 4, 0) & (0, 0, 0)\\
       &           &           &            & \mbox{if rkT=2}, (6, 10, 0)                & (4, 2, 0) & (0, 0, 0)\\
\hline
(2, 3) & (0, 1, 0) & (0, 4, 4)  & (0, 8, 7) & \mbox{if rkT=1}, (2, 10, 5) & (2, 6, 2) & (0, 2, 0) \\
       &   &     &                          & \mbox{if rkT=2}, (4, 10, 4) & (2, 4, 3) & (0, 2, 0) \\
\hline
(2, 4) & (0, 1, 0) & (0, 4, 4)  & (0, 8, 7)  & \mbox{if rkT=1}, (0, 10, 6) & (0, 8, 2)& (0, 4, 0)  \\
       &   &     &                           & \mbox{if rkT=2}, (0, 10, 5) & (0, 8, 1)& (0, 4, 0)  \\
\hline
(3, 3) & (0, 1, 0) & (0, 4, 4)  & (0, 8, 8)  & \mbox{if rkT=1}, (2, 10, 8) & (0, 6, 6)& (0, 4, 0)   \\
       &           &            &            & \mbox{if rkT=2}, (4, 10, 8) & (0, 4, 8)& (0, 4, 0)   \\
\hline
(3, 4) & (0, 1, 0) & (0, 4, 4)  & (0, 8, 8)  & (0, 10, 9)  & (0, 8, 6)   & (0, 4, 2)\\
\hline
(4, 4) & (0, 1, 0) & (0, 4, 4)  & (0, 8, 8)  & (0, 10, 10) & (0, 8, 8)   & (0, 4, 4)\\
\hline
\end{array}$
\end{center}

Note that for ${\textbf{\mbox{t}}}$ in class (iii), the rank of $T$ is always 2. So the ABC cohomology of such $X_{\textbf{\mbox{t}}}$ does not give
a finer classification than Nakamura's classification. But for ${\textbf{\mbox{t}}}$ in class (ii), the ABC cohomology may be different for $T$ with different rank. We summarize our
observation in the following. This refinement is not same as Angella's refinement of Nakamura's classification.

\begin{corollary}
We may subdivide class (ii) into 3 subclasses:
\begin{description}
\item[subclass ii.1]: rank $T$=1 and $N(T)=1$;
\item[subclass ii.2]: rank $T$=1 and $N(T)\geq 2$;
\item[subclass iii.3]: rank $T$=2.
\end{description}
\end{corollary}

\bibliographystyle{amsplain}

\end{document}